\pdfoutput=1
\documentclass{shinyart}
\usepackage[utf8]{inputenc}
\usepackage{shinybib}
\usepackage{autonum}

\usepackage[exponent-product=\cdot]{siunitx}
\usepackage{booktabs}
\usepackage{enumerate}

\newcommand{\calE}{\mathcal{E}}
\newcommand{\N}{\mathbb{N}}
\newcommand{\R}{\mathbb{R}}
\newcommand{\wkto}{\rightharpoonup}

\newcommand{\norm}[1]{\| #1 \|}
\newcommand{\set}[2]{\left\{#1:#2\right\}}
\DeclareMathOperator{\Id}{\mathrm{Id}}
\DeclareMathOperator{\proj}{\mathrm{proj}}
\DeclareMathOperator{\prox}{\mathrm{prox}}
\def\Vad{V_{\text{ad}}}

\usepackage{tikz}
\usepackage{pgfplots}
\pgfplotsset{compat=newest}
\pgfplotsset{plot coordinates/math parser=false}
\usepgfplotslibrary{fillbetween}
\usetikzlibrary{patterns}

\title{L$\scriptstyle ^1$ penalization of volumetric dose objectives in optimal control of PDEs%
    \protect\NoCaseChange{\protect\footnote{This manuscript has been authored by UT-Battelle, LLC, under Contract No. DE-AC0500OR22725 with the U.S. Department of Energy. The United States Government retains and the publisher, by accepting the article for publication, acknowledges that the United States Government retains a non-exclusive, paid-up, irrevocable, world-wide license to publish or reproduce the published form of this manuscript, or allow others to do so, for the United States Government purposes. The Department of Energy will provide public access to these results of federally sponsored research in accordance with the DOE Public Access Plan (\url{http://energy.gov/downloads/doe-public-access-plan}).}}
}
\author{%
    Richard C. Barnard\thanks{Oak Ridge National Laboratory, Oak Ridge, TN 37831, USA (\email{barnardrc@ornl.gov})} \and
    Christian Clason\thanks{Faculty of Mathematics, University Duisburg-Essen, 45117 Essen, Germany (\email{christian.clason@uni-due.de})}\and%
}

\hypersetup{
    pdftitle={L¹ penalization of volumetric dose objectives in optimal control of PDEs},
    pdfauthor={Richard C. Barnard, Christian Clason},
    pdfkeywords={optimal control, L¹ penalization, dose volume constraints, semi-smooth Newton}
}

\begin{document}

\maketitle

\begin{abstract}
    This work is concerned with a class of pde-constrained optimization problems that are motivated by an application in radiotherapy treatment planning. Here the primary design objective is to minimize the volume where a functional of the state violates a prescribed level, but prescribing these levels in the form of pointwise state constraints leads to infeasible problems. We therefore propose an alternative approach based on $L^1$ penalization of the violation that is also applicable when state constraints are infeasible.  We establish well-posedness of the corresponding optimal control problem, derive first-order optimality conditions, discuss convergence of minimizers as the penalty parameter tends to infinity, and present a semismooth Newton method for their efficient numerical solution.  The performance of this method for a model problem is illustrated and contrasted with an alternative approach based on (regularized) state constraints.
\end{abstract}

\section{Introduction}\label{sec:intro}
We consider optimal control problems governed by time-dependent linear partial differential equations in which the region where the state (or a state-dependent quantity of interest) is greater than or less than a prescribed level is to be minimized.
Such problems arise in radiotherapy treatment planning, where the aim is to deposit a radiative dose that is sufficiently strong to destroy tumor tissue while also minimizing damage to nearby healthy organs and structures.
Specifically, on the target region (the tumor), we wish the \emph{accumulated} output of the system to exceed a prescribed level $U$, while on the risk region (the healthy organs), we wish the accumulated output to not exceed a prescribed level $L$.
Due to the usual close proximity of tumors and healthy organs, it is usually not possible to satisfy these constraints over the whole region. Since a successful therapy only requires destroying (i.e., depositing a dose exceeding $U$) a sufficiently large part of the tumor, and healthy organs remain viable if a sufficiently large part remains undamaged (i.e., has a dose below $L$), volumetric conditions are used in evaluating structure survival probabilities. These are usually given in the form of dose volume histograms (DVH); see, e.g., \cite{ICRU93}.

Such performance criteria would be well-captured by an $L^0$ penalty on the violation of the prescribed limits.  However, in light of well-established difficulties associated with $L^0$ minimization, we instead propose to use $L^1$ penalty terms.
Specifically, let $\Omega\subset\R^n$ be given and let  $\omega_T,\omega_R\subset \Omega$ open, bounded, and disjoint be a target region and a risk region, respectively.
Let $\calE(y,u)=0$ denote the (time-dependent) partial differential equation with $\calE:Y\times V\rightarrow W$ for suitable Hilbert spaces $W,V,Y$, and let $C_\omega:L^2(0,T;L^1(\omega))\to L^1(\omega)$ (for either $\omega=\omega_R$ or $\omega=\omega_T$) denote the integral operator $y(t,x)\mapsto \int_0^T \chi_{\omega}(x)y(t,x)\,dt$.
For the sake of generality, we also include a quadratic tracking term with respect to a desired state $z\in L^2(Q):=L^2(0,T;L^2(\Omega))$.
We then consider problems of the form
\begin{equation}\label{eq:problem}
    \left\{\begin{aligned}
            \min_{u\in \Vad,~y\in Y} &\frac{1}{2}\norm{u}^2_{V}+\frac{\alpha}{2}\norm{y-z}^2_{L^2(Q)}+\beta_1\norm{(C_{\omega_T}y-U)^-}_{L^1(\omega_T)}+\beta_2\norm{(C_{\omega_R}y-L)^+}_{L^1(\omega_R)}\\
            \text{s.t.}\quad &\calE(y,u)=0,
    \end{aligned}\right.
\end{equation}
where $\alpha\geq 0$ and $\beta_1,\beta_2>0$, $(u)^+ = \max\{0,u\}$ and $(u)^-=\min\{0,u\}$ pointwise almost everywhere, $0<L<U$, and $\Vad\subset V$ is a set of admissible controls to be specified below.
(Spatially varying levels $U$ and $L$ are possible as well.) We call the addition of these $L^1$-penalty terms a \emph{volumetric dose penalization}, in keeping with the motivational problem from radiotherapy treatment planning. 
However, the subsequent analysis holds for more general linear PDEs. 
We also note that the analysis in the following sections can be extended to problems where $C_\omega$ takes the form of some other bounded linear functional.

As an alternative approach, one could attempt to achieve the design objectives listed above by way of pointwise state constraints on the target and risk region, i.e., by considering
\begin{equation} \label{eq:boxproblem}
    \left\{\begin{aligned} 
            \min_{u\in \Vad,~y\in Y} &\frac{1}{2}\norm{u}^2_{V}+\frac{\alpha}{2}\norm{y-z}^2_{L^2(Q)}\\
            \text{s.t.}\quad &\calE(y,u)=0,\\
                             &C_{\omega_T}y\geq U\quad\text{a.e. in } \omega_T,\\
                             &C_{\omega_R}y\leq L\quad\text{a.e. in } \omega_R.
    \end{aligned}\right.
\end{equation}
However, since the observed dose $C_\omega y$ is continuous and $L<U$, this problem is not well-posed due to the absence of feasible points if $\omega_T$ and $\omega_R$ are not separated by a strictly positive distance; as tumors can (and frequently will) occur inside vital organs, this separation does not hold in practice. (We point out that any conforming discretization of the problem will also have no feasible points.)
For instance, even in simple academic problems as in \cite{BarFraHer12}, deploying a sufficient radiative dose on the tumor is impossible unless high levels of dose are also placed on at least some portion the healthy tissue.  This becomes even more clear when additional constraints on the control are included, such as requiring the control to be a beam of a certain shape or direction.
On the other hand, if we modify the problem so that the sets are open and disjoint, the reduced cost functional for the continuous problem will not be weakly lower semi-continuous.
If one takes a cavalier approach and attempts to numerically solve the problem using, e.g., the method in \cite{ItoKun03}, one will tend to run into a numerical locking as seen in the examples below.
In contrast, \eqref{eq:problem} suffers from no such difficulty in these situations; while of course we can not expect a solution which is feasible for \eqref{eq:boxproblem} (if such a solution even exists), the theoretical existence and uniqueness of solutions is still assured and we shall see that the numerical performance is still reasonable.
In particular, we stress that \eqref{eq:problem} should not be interpreted as an exact penalization of \eqref{eq:boxproblem}.

Let us briefly comment on related literature.
A fully discretized formulation of the radiotherapy treatment planning problem as a convex linear-quadratic program was studied in \cite{SheFerOli99}.  
A treatment strategy using convexified DVH constraints was considered in \cite{ZarShaTri13}.  
However in these two works, the physics of dose deposition are discretized using precomputed beamlets; this simplification gives significant errors in dose calculation \cite{TreBog07}.
Regarding radiotherapy planning and its formulation as a PDE-constrained optimization problem, we refer to, e.g., \cite{BarFraHer12,FraHerHin12,FraHerSan10}, which use physically accurate models but do not treat DVH-based optimization strategies. 
Additionally, such models involve a significant increase in computational costs, meaning efficient optimization methods in the context of PDE-constrained problems are needed. 
Regarding $L^1$-minimization, its application to partial differential equations was first considered in the context of sparse control; see, e.g., \cite{Stadler:2007a,HerStaWac12}. $L^1$ penalization of other constraints in optimal control of PDEs was treated in \cite{Clason:2013,Clason:2014,ClaRunKunBar15}. In \cite{Gugat:2010}, an algorithm was developed which treats state-constrained problems, including \eqref{eq:boxproblem}, via a sequence of smoothed penalizations.  However, we note that here we are motivated by problems where \eqref{eq:boxproblem} does not have feasible solutions, for which such smoothing methods are not directly applicable.

This work is organized as follows. In \cref{sec:optcon}, we establish the well-posedness of \eqref{eq:problem}, derive necessary optimality conditions, and discuss the convergence of minimizers as $\beta_1,\beta_2\to\infty$. We then turn in \cref{sec:numsol} to the issue of the numerical solutions of \eqref{eq:problem} via Moreau--Yosida regularization, which allows for the use of a superlinearly convergent semismooth Newton method.  Numerical examples illustrating the behavior of the proposed approach are presented in \cref{sec:examples}.

\section{Existence and optimality conditions}\label{sec:optcon}

We first formulate \eqref{eq:problem} in reduced form. Let $V:=L^2(0,T;L^2(\omega_C))$,
\begin{equation}
    \Vad \coloneqq \set{u\in L^2(0,T;L^2(\omega_C))}{U_{\min} \leq u(t,x) \leq U_{\max} \text{ for a.e. } x\in \omega_C, t\in [0,T]}
\end{equation}
denote the admissible control set and set $Y:=L^2(0,T;L^2(\Omega))$.
We assume that for every $u\in V$, the PDE $\calE(y,u)=0$ admits a unique solutions $y\in Y$,
meaning that we can introduce a control-to-state operator 
\begin{equation}
    S:\Vad \to Y,\qquad u\mapsto y \quad\text{ solving }\quad \calE(y,u)=0.
\end{equation}
We make the assumption that $S$ is affine and bounded from $V$ to $L^2(Q)$. We note that since $S$ is affine, its Fréchet derivative $S' \eqqcolon S_0$ is given by the solution of \eqref{eq:state} with homogeneous initial and boundary conditions. 

We can thus formulate the reduced problem
\begin{equation}\label{eq:problem_red}
    \min_{u\in \Vad} \frac{1}{2}\norm{u}^2_{V}+\frac{\alpha}{2}\norm{Su-z}^2_{L^2(Q)}+\beta_1\norm{(C_{\omega_T}Su-U)^-}_{L^1(\omega_T)}+\beta_2\norm{(C_{\omega_R}Su-L)^+}_{L^1(\omega_R)}.
    \tag{$\mathcal{P}$}
\end{equation}
The final two terms take the form of integrals of convex and Lipschitz continuous integrands $g^+,g^-:\R\to \R$ with
\begin{equation}
g^+(v):= |(v-L)^+| = \begin{cases} 0 & v\leq L, \\ v-L &v\geq L,\end{cases}\qquad
g^-(v):= |(v-U)^-| = \begin{cases} U-v & v\leq U, \\ 0 &v\geq U.\end{cases}
\end{equation}
Since the bounded operators $C$ and $S$ are, respectively, linear and affine, the cost function is the sum of convex and weakly lower semi-continuous functionals, and we obtain existence of an optimal control by Tonelli's direct method. Due to the strictly convex control cost term, the optimal control is unique.
\begin{theorem}\label{thm:existence}
    For any $\alpha\geq0$ and $\beta_1,\beta_2>0$, there exists a unique minimizer $\bar u\in \Vad$ to \eqref{eq:problem_red}.
\end{theorem}
\bigskip

To derive optimality conditions, we apply the sum and chain rules of convex analysis. We first compute the subdifferentials of the volumetric dose penalty terms via the subdifferentials of the corresponding integrands $g^+,g^-$. Since both functions can be written as the maximum of two convex and differentiable functions, their convex subdifferential is given pointwise by the convex hull of the derivatives of the active functions (see, e.g., \cite[Corollary 4.3.2]{Hiriart:2001}), i.e.,
\begin{equation}\label{eq:subdiff}
\partial g^+(v)=  \begin{cases} \{0\} & v< L, \\ \{1\} &v> L, \\ [0,1] & v = L, \end{cases}\qquad\qquad
\partial g^-(v)=  \begin{cases} \{-1\} & v< U, \\ \{0\} &v> U, \\ [-1,0] & v = U. \end{cases}
\end{equation}
We also introduce the indicator function $\delta_{\Vad}: L^2(0,T;L^2(\omega_C))\to\overline \R$ in the sense of convex analysis, i.e., $\delta_{\Vad}(u) = 0$ if $u\in \Vad$ and  $\delta_{\Vad}(u) = \infty$ else. Finally, set
\begin{equation}
    G^{+}:L^2(\omega_R)\to\R,\qquad y\mapsto \int_{\omega_R} g^+(y(x))\,dx
\end{equation}
and similarly for $G^-:L^2(\omega_T)\to\R$.
We then obtain the following optimality conditions.
\begin{theorem}\label{thm:optcond}
    Let $\bar u\in\Vad$ be a minimizer of \eqref{eq:problem_red}. Then there exist $\bar \mu^+\in L^\infty(\omega_R)$ and $\bar \mu^-\in L^\infty(\omega_T)$ such that
    \begin{equation}
        \label{eq:optsys}
        \left\{\begin{aligned}
                \bar u &= \proj_{\Vad}\left(-S_0^*\left(\alpha (S\bar u -z) + \beta_1C_{\omega_R}^*\bar \mu^+ + \beta_2C^*_{\omega_T}\bar \mu^-\right)\right),\\
                \bar \mu^+ (x) &\in \partial g^+([C_{\omega_R}S\bar u](x))\quad\text{for a.e. }x\in\omega_R,\\
                \bar \mu^- (x) &\in \partial g^-([C_{\omega_T}S\bar u](x))\quad\text{for a.e. }x\in\omega_T.
        \end{aligned}\right.
        \tag{OS}
    \end{equation}
\end{theorem}
\begin{proof}
    Since \eqref{eq:problem_red} is convex, $S$ and $C_\omega$ are continuous, and all terms apart from the indicator function are finite-valued, the  sum and chain rules of convex analysis (see, e.g., \cite[Prop.~I.5.6, Prop.~I.5.7]{Ekeland:1999a}) yield the necessary optimality conditions
    \begin{equation}
        0\in \{\bar u\} + \{\alpha S_0^*(S\bar u -z)\} + \beta_1 C_{\omega_T}^*S_0^*\partial G^-(C_{\omega_T}S\bar u) + \beta_2 C_{\omega_R}^*S_0^*\partial G^+(C_{\omega_R}S\bar u) + \partial\delta_{\Vad}(\bar u).
    \end{equation}
    The fact that the subdifferential of the convex integral functional $G^+$ and $G^-$ can be computed pointwise (see, e.g. \cite[Prop.~16.50]{Bauschke:2011}) yields the second and third relation of \eqref{eq:optsys}, which also imply together with \eqref{eq:subdiff} the claimed boundedness of $\bar \mu^+$ and $\bar \mu^-$.

    Rearranging the remaining terms yields
    \begin{equation}
        -\bar u -S_0^*\left(\alpha (S\bar u -z) + \beta_1C^*_{\omega_R}\bar \mu^+ + \beta_2C^*_{\omega_T}\bar \mu^-\right) \in \partial \delta_{\Vad}(\bar u),
    \end{equation}
    which can be reformulated (denoting the second term on the left hand side by $\bar p$ for brevity)
    \begin{equation}
        \begin{aligned}
            -\bar u + \bar p \in  \partial \delta_{\Vad}(\bar u) &\Leftrightarrow
            \bar p \in \{\bar u\} +  \partial \delta_{\Vad}(\bar u) \\
            &\Leftrightarrow
            \bar u \in (\Id + \partial \delta_{\Vad})^{-1}(\bar p) = \proj_{\Vad}(\bar p),
        \end{aligned}
    \end{equation}
    using the fact that the proximal mapping of an indicator function of a convex set coincides with the (single-valued) metric projection onto this set; cf., e.g., \cite[Ex.~12.25]{Bauschke:2011}. This gives the first relation of \eqref{eq:optsys}.
\end{proof}

\bigskip

We finally address the convergence $\beta\to\infty$. Note that we do not assume the existence of a feasible solution to the state equation, which complicates the analysis and requires assuming control constraints and complete continuity of $S$, i.e., that $u_n\rightharpoonup u$ in $V$ implies $Su_n \to Su$ in $L^2(Q)$; this is in particular the case if the range of $S$ embeds compactly into $L^2(Q)$.
To simplify the presentation, we assume in the following that $\beta_2 = c\beta_1=:c\beta$ for some $c>0$.
\begin{theorem}\label{thm:opt_limit}
    Assume that $\Vad$ is bounded and that $S:V\to L^2(Q)$ is completely continuous. Then for $\beta\to \infty$, the family $\{u_\beta\}_{\beta>0}$ of solutions to \eqref{eq:problem_red} contains a subsequence converging strongly in $L^2(0,T;L^2(\omega_C))$ to a solution $\bar u\in \Vad$ of
    \begin{equation}\label{eq:problem_limit}
        \min_{u\in \Vad} \norm{(C_{\omega_T}Su-U)^-}_{L^1(\omega_T)}+c\norm{(C_{\omega_R}Su-L)^+}_{L^1(\omega_R)}.
    \end{equation}
\end{theorem}
\begin{proof}
    Since $\Vad$ is closed and assumed to be bounded, the family $\{u_\beta\}_{\beta>0}$ contains a sequence $\{u_n\}_{n\in\N}$ with $u_n\wkto \bar u\in \Vad$ with $\beta_n\rightarrow\infty$. From \cref{thm:optcond}, we obtain for every $u_n$ a corresponding pair of Lagrange multipliers 
    \begin{equation}
        \mu_n^+\in \partial G^+(C_{\omega_R}Su_n),\qquad 
        \mu_n^-\in \partial G^-(C_{\omega_T}Su_n).
    \end{equation}
    The pointwise characterization \eqref{eq:subdiff} implies that $\{\mu_n^+\}_{n\in\N}$ and $\{\mu_n^-\}_{n\in\N}$ are uniformly bounded pointwise almost everywhere. We can thus extract a further subsequence such that $\mu_n^+ \wkto \bar \mu^+$ in $L^2(\omega_R)$ and $\mu_n^- \wkto \bar \mu^-$ in $L^2(\omega_T)$.
    The complete continuity assumption on $S$ now implies $Su_n\to S\bar u$ in $L^2(Q)$ and hence $C_\omega Su_n \to C_\omega S\bar u$ in $L^2(\omega)$ (where $\omega$ is either $\omega_R$ or $\omega_T$). Hence, the weak-strong closedness of subdifferentials (see, e.g., \cite[Prop.~16.26]{Bauschke:2011}) yields that
    \begin{equation}\label{eq:optsys_limit1}
        \bar \mu^+\in \partial G^+(C_{\omega_R}S\bar u),\qquad 
        \bar \mu^-\in \partial G^-(C_{\omega_T}S\bar u).
    \end{equation}
    By Schauder's theorem and the reflexivity of $V$ and $L^2(Q)$, its adjoint $S_0^*$ is completely continuous as well. We can thus similarly deduce that
    \begin{equation}
        p_n^+:=-S_0^*(C_{\omega_R}^*\mu_n^+) \to -S_0^*(C_{\omega_R}^*\bar \mu^+)=:\bar p^+\quad\text{and}\quad p_n^-:=-S_0^*(C_{\omega_T}^*\mu_n^+) \to -S_0^*(C_{\omega_T}^*\bar\mu^+)=:\bar p^-,
    \end{equation}
    and that $S_0^*(Su_n-z)\to S_0^*(S\bar u-z)$.
    Since $\Vad$ is a closed and convex subset of $L^2(0,T;L^2(\omega_C))$, the projection $\proj_{\Vad}$ is continuous (see, e.g., \cite[Prop.~12.27]{Bauschke:2011} and use again that $\proj_{\Vad}$ coincides with the proximal mapping of the corresponding indicator function $\delta_{\Vad}$). Hence, we obtain from the first relation of \eqref{eq:optsys} that $u_n\to \bar u$ strongly as well.

    By passing to a further subsequence, we can assume that the convergence is pointwise almost everywhere. We now set $\bar  p:=\bar p^++c\bar p^-$ and carry out a pointwise inspection of \eqref{eq:optsys}.
    \begin{enumerate}[(i)]
        \item $\bar p(x)>0$: In this case, 
            \begin{equation}
                [-\alpha S_0^*(Su_n-z) + \beta_n ( p_n^+ + c  p_n^-)](x) \to \infty,
            \end{equation}
            which implies that there exists an $N\in\N$ such that
            \begin{equation}
                \bar u(x) = u_n(x) = U_{\min} \qquad\text{for all }n>N.
            \end{equation}
        \item $\bar p(x)<0$: In this case,
            \begin{equation}
                [-\alpha S_0^*(Su_n-z) + \beta_n ( p_n^+ + c  p_n^-)](x) \to -\infty,
            \end{equation}
            and we similarly obtain the existence of an $N\in\N$ such that
            \begin{equation}
                \bar u(x) = u_n(x) = U_{\max} \qquad\text{for all }n>N.
            \end{equation}
        \item $\bar p(x)=0$: In this case, we can only conclude that $\bar u(x)\in [U_{\min},U_{\max}]$.
    \end{enumerate}
    We thus conclude that
    \begin{equation}\label{eq:optsys_limit2}
        \bar u(x) \in \begin{cases} \{U_{\min}\} & \text{if }\bar  p(x) <0,\\
            [U_{\min},U_{\max}] & \text{if }\bar  p(x) = 0,\\
            \{U_{\max}\} & \text{if }\bar  p(x) >0.
        \end{cases}
    \end{equation}

    Now consider \eqref{eq:problem_limit}. Proceeding as in the proof of \cref{thm:optcond}, we deduce for any solution $\bar u$ the existence of $\bar\mu^+,\bar \mu^-$ satisfying \eqref{eq:optsys_limit1} such that 
    \begin{equation}
        \bar p:= -S_0^*(C_{\omega_R}\bar \mu^+ + c C_{\omega_T}\bar \mu^-) \in \partial\delta_{\Vad}(\bar u),
    \end{equation}
    which can be reformulated as
    \begin{equation}
        \bar u = \proj_{\Vad}(\bar u+\bar p).
    \end{equation}
    By pointwise inspection, this is equivalent to \eqref{eq:optsys_limit2}. Hence, the limit of $\{u_n\}_{n\in\N}$ satisfies the optimality conditions for the convex problem \eqref{eq:problem_limit} and is therefore a minimizer.
\end{proof}
Note that \eqref{eq:problem_limit} does not coincide with \eqref{eq:boxproblem}, which may not admit a solution. However, if there exists a solution to \eqref{eq:boxproblem}, it is obviously also a solution to \eqref{eq:problem_limit}. In fact, under this assumption, standard arguments show weak subsequential convergence of $u_\beta$ to a solution to \eqref{eq:boxproblem}; for the sake of completeness, we give a full proof here.
\begin{proposition}\label{thm:opt_limit_feas}
    Assume that \eqref{eq:boxproblem} admits a solution $\hat u$. Then for $\beta\to \infty$, the family $\{u_\beta\}_{\beta>0}$ of solutions to \eqref{eq:problem_red} contains a subsequence converging weakly in $L^2(0,T;L^2(\omega_C))$ to $\hat u$.
\end{proposition}
\begin{proof}
    By optimality of $u_\beta$ for any $\beta>0$ and feasibility of $\hat u$, we have that 
    \begin{equation}\label{eq:limit_feas1}
        \begin{aligned}[t]
            \frac{1}{2}\norm{u_\beta}^2_{V}
            &\leq \frac{1}{2}\norm{u_\beta}^2_{V}+\frac{\alpha}{2}\norm{Su_\beta -z}^2_{L^2(Q)}+\beta\norm{(C_{\omega_T}Su_\beta-U)^-}_{L^1(\omega_T)}\\
            \MoveEqLeft[-12]+c\beta\norm{(C_{\omega_R}Su_\beta-L)^+}_{L^1(\omega_R)}\\
            &\leq  \frac{1}{2}\norm{\hat u}^2_{V}+\frac{\alpha}{2}\norm{S\hat u -z}^2_{L^2(Q)}.
        \end{aligned}
    \end{equation}
    Hence, the family $\{u_\beta\}_{\beta>0}$ is bounded in $V$ and therefore contains a sequence $\{u_n\}_{n\in\N}\subset \Vad$ with $\beta_n\to\infty$ and $u_n\wkto u^*$ for some $u^*\in V$. Since $\Vad$ is convex and closed, it follows that $u^*\in \Vad$ as well. 
    Similarly, we obtain from \eqref{eq:limit_feas1} that as $\beta\to\infty$,
    \begin{equation}
        \norm{(C_{\omega_T}Su_\beta-U)^-}_{L^1(\omega_T)}+c\norm{(C_{\omega_R}Su_\beta-L)^+}_{L^1(\omega_R)}
        \leq \frac1\beta\left(\frac{1}{2}\norm{\hat u}^2_{V}+\frac{\alpha}{2}\norm{S\hat u -z}^2_{L^2(Q)}\right) \to 0.
    \end{equation}
    By continuity of $S$ and $C_\omega$, and possibly after passing to a further subsequence such that $u_n\to u^*$ pointwise almost everywhere, we deduce from this that
    \begin{equation}
        C_{\omega_T}Su^* \geq U \quad\text{a.e. in }\omega_T,\qquad
        C_{\omega_R}Su^* \leq L \quad\text{a.e. in }\omega_R,
    \end{equation}
    and hence that $u^*$ is feasible for \eqref{eq:boxproblem} with $y^*=Su^*$.
    Continuity of $S$, weak lower semi-continuity of norms, and optimality of $u_n=u_{\beta_n}$ then implies that 
    \begin{equation}
        \begin{aligned}
            \frac{1}{2}\norm{u^*}^2_{V}+\frac{\alpha}{2}\norm{Su^* -z}^2_{L^2(Q)} 
            &\leq \liminf_{n\to\infty}  \frac{1}{2}\norm{u_n}^2_{V}+\frac{\alpha}{2}\norm{Su_n -z}^2_{L^2(Q)}\\
            &\leq \limsup_{n\to\infty}  \frac{1}{2}\norm{u_n}^2_{V}+\frac{\alpha}{2}\norm{Su_n -z}^2_{L^2(Q)} + \beta_n\norm{(C_{\omega_T}Su_n-U)^-}_{L^1(\omega_T)}\\
            \MoveEqLeft[-15]+c\beta_n\norm{(C_{\omega_R}Su_n-L)^+}_{L^1(\omega_R)}\\
            &\leq \frac{1}{2}\norm{\tilde u}^2_{V}+\frac{\alpha}{2}\norm{S\tilde u -z}^2_{L^2(Q)} 
        \end{aligned}
    \end{equation}
    for any feasible $\tilde u \in \Vad$, i.e., $u^*$ is a minimizer of \eqref{eq:boxproblem}. Since \eqref{eq:boxproblem} is strictly convex, the minimizer---if it exists---must be unique, which yields $u^*=\hat u$.
\end{proof}
Since the penalization of feasible constraints is not the focus of this paper, we omit further analysis of this case and refer instead to, e.g., \cite{Gugat:2010,Gugat:2010b}.

\section{Numerical solution}\label{sec:numsol}

In order to solve \eqref{eq:regoptsys}, we proceed similarly to \cite{Clason:2011} and use a semismooth Newton method applied to a Moreau--Yosida regularization of \eqref{eq:optsys}.

\subsection{Moreau--Yosida regularization}\label{sec:MYreg}

To compute the Moreau--Yosida regularization, we replace $\partial g^+$ for $\gamma>0$ by
\begin{equation}\label{eq:my_reg}
    \partial g^+_\gamma(v) := (\partial g^+)_\gamma(v) := \frac{1}{\gamma} \left(v-\prox_{\gamma g^+}(v)\right),
\end{equation}
where
\begin{equation}
    \prox_{\gamma g^+}(v) := \arg\min_{w\in \R}\frac1{2\gamma}|w-v|^2 + g^+(w) = \left(\Id + \gamma\partial g^+\right)^{-1}(v)
\end{equation}
is the proximal mapping of $g^+$, which in Hilbert spaces coincides with the resolvent of $\partial g^+$; see, e.g., \cite[Prop.~16.34]{Bauschke:2011}.  Note that the proximal mapping and
thus the Moreau--Yosida regularization of a proper and convex functional is always single-valued and Lipschitz continuous; see, e.g., \cite[Corollary~23.10]{Bauschke:2011}.

We begin by calculating the proximal mapping of $g^+$, proceeding as in \cite{Clason:2014}.
For given $\gamma>0$ and $v\in\R$, the resolvent $w:=(\Id+\gamma\partial g^+)^{-1}(v)$ is characterized by the subdifferential inclusion
\begin{equation}\label{eq:resolvent}
    v\in (\Id+\gamma\partial g^+)(w) = \{w\}+\gamma\partial g^+(w).
\end{equation}
We now follow the case discrimination in the characterization \eqref{eq:subdiff} of the subdifferential.
\begin{enumerate}[(i)]
    \item $w<L$: In this case, we have that $v = w <L$.
    \item $w>L$: In this case, we have that $v = w+\gamma < L + \gamma$, i.e., $w=v-\gamma$.
    \item $w=L$: In this case, we have that $v \in w +\gamma[0,1] = [L,L+\gamma]$.
\end{enumerate}
Since these cases yield a complete and disjoint case distinction for $v$, we obtain
\begin{equation}
    \prox_{\gamma g^+}(v)=
    \begin{cases}
        v & \text{if } v<L, \\
        L & \text{if } v\in [L,L+\gamma],\\
        v-\gamma & \text{if } v>L+\gamma.
    \end{cases}
\end{equation}
Inserting this into the definition of the Moreau--Yosida regularization gives
\begin{equation}
    \partial g^+_\gamma(v) =
    \begin{cases}
        0 & \text{if } v< L, \\
        \frac1\gamma (v-L) & \text{if } v\in [L,L+\gamma],\\
        1 & \text{if } v>L+\gamma.
    \end{cases}
\end{equation}
Proceeding similarly for $g^-$, we find that
\begin{equation}
    \prox_{\gamma g^-}(v)=
    \begin{cases}
        v+\gamma & \text{if } v<U-\gamma, \\
        U & \text{if } v\in [U-\gamma,U],\\
        v & \text{if } v>U,
    \end{cases}
\end{equation}
and hence
\begin{equation}
    \partial g^-_\gamma(v) =
    \begin{cases}
        -1 & \text{if } v< U-\gamma, \\
        \frac1\gamma (v-U) & \text{if } v\in [U-\gamma,U],\\
        0 & \text{if } v>U.
    \end{cases}
\end{equation}

Replacing the subdifferentials with their regularizations in \eqref{eq:optsys}, we arrive at the regularized system
\begin{equation}
    \label{eq:regoptsys}
    \left\{\begin{aligned}
            u_\gamma &= \proj_{\Vad}\left(-S_0^*\left(\alpha (Su_\gamma-z) + \beta_1 C^*_{\omega_R}\mu^+_\gamma  + \beta_2C^*_{\omega_T}\mu^-_\gamma\right)\right),\\
            \mu^+_\gamma (x) &= \partial g^+_\gamma([C_{\omega_R}Su_\gamma](x)),\\
            \mu^-_\gamma (x) &= \partial g^-_\gamma([C_{\omega_T}Su_\gamma](x)).
    \end{aligned}\right.
    \tag{OS$_\gamma$}
\end{equation}

\begin{theorem}\label{thm:reg_existence}
    For every $\gamma>0$, there exists $(u_\gamma,\mu^+_\gamma,\mu^+_\gamma)$ satisfying \eqref{eq:regoptsys}.
\end{theorem}
\begin{proof}
    We use the fact that $\partial g^+_\gamma(v)$ is the derivative of the (convex and lower semi-continuous) Moreau envelope
    \begin{equation}
        g^+_\gamma(v) := g^+(\prox_{\gamma g^+}(v)) + \frac1{2\gamma} |v-\prox_{\gamma g^+}(v)|^2,
    \end{equation}
    see, e.g., \cite[Remark 12.24, Proposition 12.29]{Bauschke:2011}; a similar statement holds for $\partial g_\gamma^-$. Hence, \eqref{eq:regoptsys} are the necessary optimality conditions of the convex minimization problem
    \begin{equation}\label{eq:regprob}
        \min_{u\in \Vad} \frac{1}{2}\norm{u}^2_{V}+\frac{\alpha}{2}\norm{Su-z}^2_{L^2(Q)}+\beta_1 \int_{\omega_T} g^-_\gamma(C_{\omega_T}Su)\,dx + \beta_2 \int_{\omega_T} g^+_\gamma(C_{\omega_R}Su)\,dx,
    \end{equation}
    which admits a unique solution.
\end{proof}

\begin{remark}
    The Moreau envelopes of $g^+$ and $g^-$ are given by
    \begin{equation}
        g^+_\gamma(v) =
        \begin{cases}
            0 & \text{if } v< L, \\
            v-L-\frac\gamma2 & \text{if } v>L+\gamma,\\
            \frac1{2\gamma} (v-L)^2 & \text{if } v\in [L,L+\gamma],
        \end{cases}
        \qquad
        g^-_\gamma(v) =
        \begin{cases}
            U-v-\frac\gamma2 & \text{if } v<U-\gamma,\\
            0 & \text{if } v> U, \\
            \frac1{2\gamma} (v-U)^2 & \text{if } v\in [U-\gamma,U].
        \end{cases}
    \end{equation}
    The Moreau--Yosida regularization of the dose penalty is thus related to the well-known Huber-regularization of the $L^1$ norm.
\end{remark}
We conclude this section by noting that solutions to the regularized system \eqref{eq:regoptsys}  converge weakly up to a subsequence to solutions of the original optimality system \eqref{eq:optsys}.
\begin{theorem}
    The family $\{(u_{\gamma},\mu^+_{\gamma},\mu^-_{\gamma})\}_{\gamma>0}$ contains a sequence $\{(u_{\gamma_n},\mu^+_{\gamma_n},\mu^-_{\gamma_n})\}_{n\in\N}$ converging weakly to a solution $(\bar{u},\bar{\mu}^+,\bar{\mu}^-)$ of \eqref{eq:optsys}.
\end{theorem}
\begin{proof}
    The proof follows largely that of \cite[Prop.~2.5]{Clason:2014}. We note that for any $u\in V$, $\partial g_\gamma^+(u(x))$ and $\partial g_\gamma^-(u(x))$ are bounded almost everywhere, implying that $\{\mu^+_{\gamma}\}_{\gamma>0},\{\mu^-_{\gamma}\}_{\gamma>0}$ are bounded. As $S_0^*$ and $C_\omega^*$ are bounded linear operators, the family
    \begin{equation}
        \{p_\gamma\}_{\gamma>0}:=\left\{S_0^*\left(\alpha (S u_{\gamma} -z) + \beta_1C^*_{\omega_R}\mu_{\gamma}^+ + \beta_2C^*_{\omega_T} \mu_{\gamma}^-\right)\right\}_{\gamma >0 }
    \end{equation}
    is bounded in $V$. This in turn implies the boundedness of $\{u_{\gamma}\}_{\gamma>0}$. Hence, there exists a subsequence converging weakly to $(\hat p, \hat{u},\hat{\mu}^+,\hat{\mu}^-)$. As $g^+$ and $g^-$ are convex, and therefore $\partial g^+(u(x))$ and $\partial g^-(u(x))$ are maximal monotone for every $u\in V$ and almost every $x\in \omega_C$, we have by \cite[Lemma 1.3\,(e)]{BreCraPaz70} that $\hat{\mu}^+,\hat{\mu}^-$ satisfy the second and third relations of \eqref{eq:optsys}. The first relation follows similarly, using that $u= \proj_{\Vad}(p)$ is equivalent to the subdifferential inclusion $-u +p \in\partial\delta_{\Vad}(u)$; see the proof of \cref{thm:optcond}. 
\end{proof}

\subsection{Semismooth Newton method}\label{sec:SSN}

The solution to \eqref{eq:regoptsys} can be computed using a semismooth Newton method \cite{Kunisch:2008a,Ulbrich:2011}. Since $h^+_\gamma := \partial g^+_\gamma$ and $h^-_\gamma:=\partial g^-_\gamma$ are globally Lipschitz continuous and piecewise differentiable, they are Newton-differentiable with Newton derivatives given by
\begin{equation}
D_N h^+_\gamma (v) = \begin{cases} \frac 1\gamma & \text{if } v\in [L,L+\gamma],\\ 0 &\text{else},\end{cases}\qquad
D_N h^-_\gamma (v) = \begin{cases} \frac 1\gamma & \text{if } v\in [U-\gamma,U],\\ 0 &\text{else},\end{cases}\qquad
\end{equation}
see, e.g.,  \cite[Proposition 2.26]{Ulbrich:2011}.
Similarly, $\proj_{\{[U_{\min},U_{\max}]\}}(v)$ is Newton-differentiable with Newton derivative given by
\begin{equation}
D_N \proj_{\{[U_{\min},U_{\max}]\}} (v) = \begin{cases}1 & \text{if } v\in [U_{\min},U_{\max}],\\ 0 &\text{else}.\end{cases}\qquad
\end{equation}
This implies that the corresponding superposition operators
$H_\gamma^\pm :L^p(\omega)\to L^2(\omega)$ and $\proj_{\Vad}:L^p(0,T;L^p(\omega_C))\to L^2(0,T;L^2(\omega_C))$ are semismooth,
with Newton derivatives given pointwise by, e.g.,
\begin{equation}
[D_N H^+_\gamma(y)](x) = \frac1\gamma [\chi^+(y)](x)\coloneqq \begin{cases} \frac 1\gamma & \text{if } y(x)\in [L,L+\gamma],\\ 0 &\text{else},\end{cases}
\end{equation}
and $D_N\proj_{\Vad}(y) = \chi_{\Vad}(y)$; see, e.g., \cite[Example 8.12]{Kunisch:2008a} or \cite[Theorem 3.49]{Ulbrich:2011}.

To apply a semismooth Newton to \eqref{eq:regoptsys}, we rewrite it by  eliminating $\mu_\gamma^+,\mu_\gamma^-$ as
\begin{equation}\label{eq:regoptsys_ref}
    u_\gamma - \proj_{\Vad}\left(-S_0^*\left(\alpha (Su_\gamma-z) + \beta_1 C^*_{\omega_R}H^+_\gamma(C_{\omega_R}Su_\gamma) + \beta_2 C^*_{\omega_T}H^-_\gamma(C_{\omega_T}Su_\gamma)\right)\right)=0.
\end{equation}
We further assume that the range of $S$ (and hence of $S_0^*$) is contained (not necessarily compactly) in $L^p(0,T;L^p(\Omega))$ for some $p>2$, which also implies that the range of $C_\omega S$ is contained in $L^p(\omega)$ for any subdomain $\omega\subset \Omega$.
By the sum and chain rules of Newton derivatives (see, e.g., \cite[Theorem 3.69]{Ulbrich:2011}) it then follows that \eqref{eq:regoptsys_ref}---taken as an operator equation $T(u)=0$ for $T:L^2(0,T;L^2(\omega_C)) \to L^2(0,T;L^2(\omega_C))$---is semismooth.

In order to establish the invertibility of the Newton step $D_N T(u^k)\delta u = -T(u^k)$, i.e.,
\begin{equation} \label{eq:Newtsys}
    \left(\Id + \chi_{\Vad}(-F(u^k))D_NF(u^k)\right)\delta u =
    - \left(u^k - \proj_{\Vad}(-F(u^k))\right)
\end{equation}
with
\begin{align}
    F(u^k)&:=S_0^*\left(\alpha (Su^k-z) + \beta_1 C^*_{\omega_R}H^+_\gamma(C_{\omega_R}Su^k) + \beta_2 C^*_{\omega_T}H^-_\gamma(C_{\omega_T}Su^k)\right),\\
    D_NF(u^k) &= S_0^*\left(\alpha + \tfrac{\beta_1}\gamma  C^*_{\omega_R}\chi^+(C_{\omega_R}S_0u^k) C_{\omega_R} + \tfrac{\beta_2}{\gamma}  C^*_{\omega_T}\chi^-(C_{\omega_T}S_0u^k)C_{\omega_T}\right)S_0,
\end{align}
we note that $D_NF(u^k)=S^*_0AS_0$ for a positive and self-adjoint linear operator $A$, and thus that $D_NF(u^k)$ is positive and self-adjoint for every $u^k$.  We recall the following result:
\begin{lemma}[Corrected Corollary 3.5 of \cite{HerStaWac12,HerStaWac15}]
    If $A$ and $B$ are positive, self-adjoint operators on a Hilbert space, then $\sigma(AB)\subset[0,\infty)$.
\end{lemma}
As $\chi_{\Vad}(u^k)$ is positive for any $u^k$, we have that $\sigma\left(\chi_{\Vad}(u^k) D_NF(u^k)\right)\subset[0,\infty)$ for any $k\in\N$, and therefore $\Id+\chi_{\Vad}(u^k) D_NF(u^k)$ is uniformly invertible.

By standard arguments, the uniform invertibility of the left-hand side operator in \eqref{eq:Newtsys} together with the Newton-differentiability implies local superlinear convergence of the corresponding semismooth Newton method to a solution to \eqref{eq:regoptsys} for each $\gamma>0$; see, e.g., \cite[Thm.~8.16]{Kunisch:2008a}, \cite[Chap.~3.2]{Ulbrich:2011}.

For given $h$, the application of the Newton derivative $D_N F(u^k)h$ can be computed by solving the linearized state equation \eqref{eq:state}, applying pointwise operations, and then solving the linearized adjoint equation. Hence, the update $\delta u$ solving the semismooth Newton step $D_N T(u^k)\delta u = - T(u^k)$ can be computed by a matrix-free Krylov method. 
To account for the local convergence of Newton methods, we embed the semismooth Newton method within a homotopy strategy for $\gamma$,
where we start with a large $\gamma$ which is successively reduced, taking the previous solution as starting point.
Furthermore, we include a backtracking line search based on the residual norm $\norm{T(u^{k+1})}$ to improve robustness.
Our Python implementation of this approach, which was used to generate the results below, can be downloaded from \url{https://www.github.com/clason/dvhpenalty}.

\section{Numerical examples}\label{sec:examples}

To illustrate the performance of the proposed approach, we compare the effects of the volumetric dose penalty with the corresponding state constraints for a simple test problem. For the sake of illustration, we consider in this section the partial differential equation
\begin{equation}\label{eq:state}
    \calE(y,u) = y_t - c\Delta y - E_{\omega_C}u
\end{equation}
for some $c>0$ together with initial conditions $y(0)=y_0\in L^2(\Omega)$ and homogeneous Dirichlet boundary conditions, where $E_{\omega_C}:V:=L^2(0,T;L^2(\omega_C))\to L^2(0,T;L^2(\Omega))$ denotes the extension by zero operator and $\omega_C\subset \Omega$ is the bounded control region. 
Let $\Omega=[-1,1]$, $T=1$, and $c=0.01$. 
We choose the target, risk and control regions as $\omega_T:=[-0.45,0.45]\setminus[-0.2,0.2] $, $\omega_R:=[-0.7,-0.55] \cup [0.55,0.7] \cup [-0.2,0.2]$, and $\omega_C=\Omega$, respectively.
We further let $U=0.5$, $L=0.2$, $U_{\min}=0$, and $U_{\max}=2$. 
Finally, we set $\alpha = 0$ and accordingly do not require a target $z$. To illustrate the influence of the dose penalty parameters $\beta_1$ (on the target region) and $\beta_2$ (on the risk region), we set $\beta_1=\tilde \beta_1 |\omega_T|^{-1}$ and $\beta_2 = \tilde \beta_2 |\omega_R|^{-1}$ for $\tilde \beta_1,\tilde \beta_2 \in\{10^5,10^6\}$, where $|\omega_T|=0.5$ and $|\omega_R|=0.7$ denote the Lebesgue measure of the target and risk region, respectively; this scaling ensures that if $\tilde \beta_1=\tilde \beta_2$, both objectives are given equal weight.
We also solve \eqref{eq:boxproblem} with identical parameters (where applicable) by solving a sequence of Moreau--Yosida-regularized problems (which coincide with a quadratic penalization of the state constraints with penalty parameter $\gamma^{-1}$) via a semismooth Newton method as outlined in \cite{ItoKun03}.

In the following, a spatial discretization with $256$ nodes and $256$ time steps are used.
In order to compute for each $\gamma$ a minimizer of the Moreau--Yosida regularization of \eqref{eq:problem} and \eqref{eq:boxproblem}, we use a maximum of $100$ semismooth Newton iterations; each Newton step is computed using GMRES with a maximum of $3000$ iterations.
We initialize $\gamma$ as $\gamma_0:=\max\{\beta_1,\beta_2\}$ for \eqref{eq:problem} and as $\gamma_0:=1$ for \eqref{eq:boxproblem}, and in both cases reduce $\gamma$ by a factor of $2$ as long as the Newton method converges until $\gamma$ reaches $10^{-10}\gamma_0$ for solving \eqref{eq:problem} and $10^{-7}\gamma_0$ for solving \eqref{eq:boxproblem}, respectively.
The convergence criterion used for the Newton iterations is a reduction to below $10^{-6}$ of the norm of the optimality system.
The results for solving \eqref{eq:boxproblem} and \eqref{eq:problem} are given in \cref{fig:boxdose} and \crefrange{fig:dose55}{fig:dose66}, respectively, for the last value of $\gamma$ (noted below) for which the semismooth Newton method converged.  In each case, the dose volume histogram shows the fraction of the area of the regions $\omega_R$ and $\omega_T$ where the dose $C_{\omega_R}y$ and $C_{\omega_T}y$ is \emph{at least} that level (i.e., the objective is to minimize the area of the shaded regions between the dotted lines and the curves).

\definecolor{mycolor1}{rgb}{0.00000,0.44700,0.74100}%
\definecolor{mycolor2}{rgb}{0.85000,0.32500,0.09800}%
\definecolor{mycolor3}{rgb}{0.92900,0.69400,0.12500}%
\definecolor{mycolor4}{rgb}{0.49400,0.18400,0.55600}%
\definecolor{mycolor5}{rgb}{0.46600,0.67400,0.18800}%
\definecolor{mycolor6}{rgb}{0.30100,0.74500,0.93300}%
\definecolor{mycolor7}{rgb}{0.63500,0.07800,0.18400}%
\begin{figure}[t]
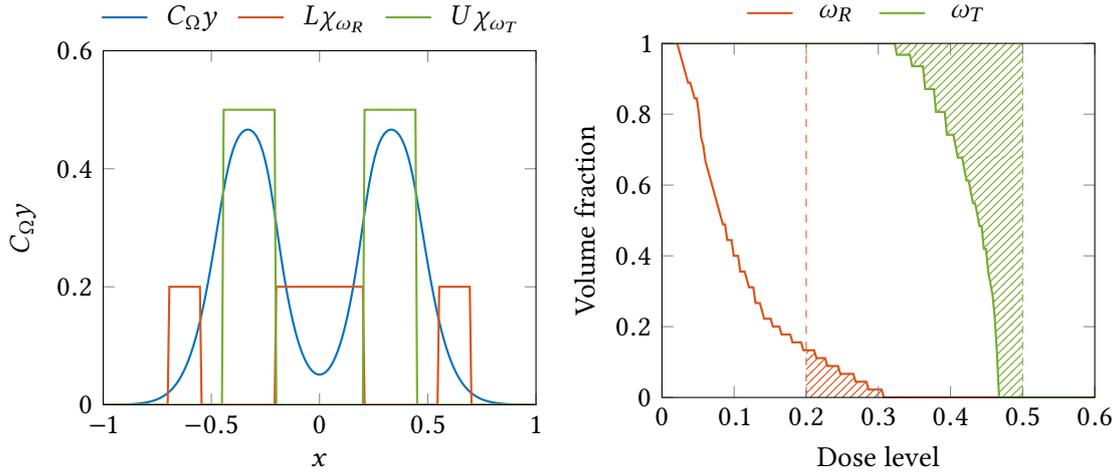

    \centering
    \begin{subfigure}[t]{0.495\textwidth}
        \input{Dose_constraint.tex}
        \caption{Final dose $C_\Omega y$,  risk level $L\chi_{\omega_R}$, target level $U\chi_{\omega_T}$}
    \end{subfigure}
    \hfill
    \begin{subfigure}[t]{0.495\textwidth}
        \input{DVH_constraint.tex}
        \caption{Dose volume histogram for risk region $\omega_R$ and target region $\omega_T$}
    \end{subfigure}
    \caption{Dose information for dose-constrained problem \eqref{eq:boxproblem}}
    \label{fig:boxdose}
\end{figure}
\begin{figure}[p]
    \centering
    \begin{subfigure}[t]{0.495\textwidth}
        \input{Dose_55.tex}
        \caption{Final dose $C_\Omega y$,  risk level $L\chi_{\omega_R}$, target level $U\chi_{\omega_T}$}
    \end{subfigure}
    \hfill
    \begin{subfigure}[t]{0.495\textwidth}
        \input{DVH_55.tex}
        \caption{Dose volume histogram for risk region $\omega_R$ and target region $\omega_T$}
    \end{subfigure}
    \caption{Dose information for dose-penalized problem \eqref{eq:problem} with $\tilde\beta_1=10^5$, $\tilde\beta_2=10^5$}
    \label{fig:dose55}
\end{figure}
\begin{figure}[p]
    \centering
    \begin{subfigure}[t]{0.495\textwidth}
        \input{Dose_65.tex}
        \caption{Final dose $C_\Omega y$,  risk level $L\chi_{\omega_R}$, target level $U\chi_{\omega_T}$}
    \end{subfigure}
    \hfill
    \begin{subfigure}[t]{0.495\textwidth}
        \input{DVH_65.tex}
        \caption{Dose volume histogram for risk region $\omega_R$ and target region $\omega_T$}
    \end{subfigure}
    \caption{Dose information for dose-penalized problem \eqref{eq:problem} with $\tilde\beta_1=10^6$, $\tilde\beta_2=10^5$}
    \label{fig:dose65}
\end{figure}
\begin{figure}[p]
    \centering
    \begin{subfigure}[t]{0.495\textwidth}
        \input{Dose_56.tex}
        \caption{Final dose $C_\Omega y$,  risk level $L\chi_{\omega_R}$, target level $U\chi_{\omega_T}$}
    \end{subfigure}
    \hfill
    \begin{subfigure}[t]{0.495\textwidth}
        \input{DVH_56.tex}
        \caption{Dose volume histogram for risk region $\omega_R$ and target region $\omega_T$}
    \end{subfigure}
    \caption{Dose information for dose-penalized problem \eqref{eq:problem} with $\tilde\beta_1=10^5$, $\tilde\beta_2=10^6$}
    \label{fig:dose56}
\end{figure}
\begin{figure}[p]
    \centering
    \begin{subfigure}[t]{0.495\textwidth}
        \input{Dose_66.tex}
        \caption{Final dose $C_\Omega y$,  risk level $L\chi_{\omega_R}$, target level $U\chi_{\omega_T}$}
    \end{subfigure}
    \hfill
    \begin{subfigure}[t]{0.495\textwidth}
        \input{DVH_66.tex}
        \caption{Dose volume histogram for risk region $\omega_R$ and target region $\omega_T$}
    \end{subfigure}
    \caption{Dose information for dose-penalized problem \eqref{eq:problem} with $\tilde\beta_1=10^6$, $\tilde\beta_2=10^6$}
    \label{fig:dose66}
\end{figure}

We first note that the solution of the regularized state-constrained problem, shown in \cref{fig:boxdose} for the final $\gamma\approx 1.22\cdot 10^{-4}$, gives poor results.
This is not unexpected: the problem \eqref{eq:boxproblem} is clearly infeasible; we see that for all $\gamma$, we have $C_{\omega_T}y<U$ everywhere while $C_{\omega_R}y>L$ on around $13\%$ of the risk region.
This means that one primary design objective---exceeding the minimal dose $U$ on $\omega_T$---is not achieved at all.

Meanwhile, the solution to the regularized dose-penalized problem for $\tilde\beta_1=\tilde \beta_2 = 10^5$ and final $\gamma/\gamma_0\approx 4.66\cdot 10^{-10}$ shown in \cref{fig:dose55} is clearly superior: a significant portion ($45\%$) of the target region $\omega_T$ has at least a dose of $U$, while the area where $C_{\omega_R}y>L$ is slightly smaller ($11\%$). Increasing $\tilde\beta_1$ to $10^6$ (see \cref{fig:dose65}, with final $\gamma/\gamma_0 \approx 1.16\cdot10^{-10}$) further improves the dose coverage on the target ($81\%$), but does so at the expense of increased violation of the dose constraint on the risk region ($22\%$ instead of $11\%$). Conversely, increasing $\tilde \beta_2$ to $10^6$  while keeping $\tilde \beta_1$ at $10^5$ (see \cref{fig:dose56}, with final $\gamma/\gamma_0\approx 1.16\cdot 10^{-10}$) reduces the dose violation on the risk region to $2\%$, but the coverage on the target is now only $42\%$. Finally, increasing both $\tilde\beta_1$ and $\tilde \beta_2$ to $10^6$ (see \cref{fig:dose66}, with final $\gamma/\gamma_0\approx 4.66\cdot 10^{-10}$) yields a dose coverage on the target of $84\%$, while the dose violation on the risk region is still only $11\%$.
Thus, in contrast to state constraints, the penalization of the dose violation is able to balance the competing objectives.

This comparison is more evident in \cref{tab:hardcaseBox} and \cref{tab:dose},
where we report for selected values of $\gamma$ the number of Newton steps needed as well as the fraction of the area of $\omega_T$ where the resulting $C_{\omega_T}y$ is below $U$ and the fraction of the area of $\omega_R$ where the resulting $C_{\omega_R}y$ is above $L$.
We see that for regularized state constraints, the regularization approach becomes significantly more difficult for even modestly small $\gamma$ while failing to give reasonable performance, which is again not surprising since the limit problem is infeasible.
In comparison, significantly fewer Newton iterations are required for the dose penalty, and the solutions to \eqref{eq:regoptsys} give better performance for each $\gamma$. 
(Here it should be pointed out that because the ratio $\beta/\gamma$ enters into the Newton system for the dose penalization, the values of $\gamma$ are not directly comparable to the case of state constraints, where $1/\gamma$ enters into the Newton system.)
It can also be seen that, at least in this configuration, putting more weight on the target region increases the difficulty of the problem significantly at the end of the homotopy loop, while putting more or equal weight on the risk region requires fewer Newton iterations for each value of $\gamma$.
We note that after $\gamma/\gamma_0 = 1.86\cdot 10^{-9}$, the volume fraction where the dose exceeds $L$ (respectively, is below $U$) on $\omega_R$ (resp. $\omega_T$) is unchanged for several iterations in the homotopy (not all of which are shown in the tables).
Finally, we remark that for $\tilde{\beta}_1=\tilde{\beta}_2 \in\{10^{7},~10^8,10^9\}$, the final volume fractions where the dose on $\omega_R$ is above $L$ is consistently $11\%$, and where the dose on $\omega_T$ is below $U$ is consistently $9.7\%$; for  $\tilde{\beta}_1=\tilde{\beta}_2 \geq 10^{10}$, the final volume fractions are consistently $17.8\%$ and $6.5\%$, respectively, which is to be expected in light of \cref{thm:opt_limit}.

\begin{table}[b]
    \caption {Results for dose-constrained problem \eqref{eq:boxproblem}: number of SSN steps, volume fraction (as percentage) for risk and target regions for different values of $\gamma$ ($*$ denotes failure to converge)}
    \label{tab:hardcaseBox}
    \centering
    \begin{tabular}{ccccccc}
        \toprule
        $\gamma/\gamma_0$                & $1.25\cdot 10^{-1}$ & $9.77\cdot 10^{-4}$ & $4.88\cdot 10^{-4}$ & $2.44\cdot 10^{-4}$ & $1.22\cdot 10^{-4}$ & $6.10 \cdot 10^{-5}$\\
        \midrule
        \#SSN                   & $1$                 & $1$                 & $1$                 & $6$                 & $24$                & $*$  \\
        \% $\omega_R$ above $L$ & $0$                 & $0$                 & $6.67$              & $11.11$              & $13.33$              & $*$ \\
        \% $\omega_T$ below $U$ & $100$                 & $100$                 & $100$                 & $100$                 & $100$                 & $*$  \\
        \bottomrule
    \end{tabular}
\end{table}
\begin{table}
    \caption {Results for dose-penalized problem \eqref{eq:problem}: number of SSN steps, volume fraction (as percentage) for risk and target regions for different values of $\gamma$ ($*$ denotes failure to converge)}\label{tab:dose}
    \begin{subtable}{\linewidth}
        \caption{$\tilde\beta_1=10^5$, $\tilde\beta_2=10^5$}\label{tab:dose:55}
        \centering
        \begin{tabular}{ccccccc}
            \toprule 
            $\gamma/\gamma_0$                & $1.56\cdot 10^{-2}$ & $1.22\cdot 10^{-4}$ & $1.91\cdot 10^{-6}$ & $1.49\cdot 10^{-8}$ & $1.86\cdot 10^{-9}$ & $1.16 \cdot 10^{-10}$\\
            \midrule
            \#SSN                   & $1$              & $1$              & $3$              & $3$              & $4$              & $*$\\
            \% $\omega_R$ above $L$ & $0$           & $0$           & $13.33$           & $13.33$           & $11.11$              & $*$ \\
            \% $\omega_T$ below $U$ & $100$           & $100$           & $100$           & $54.84$           & $54.84$              & $*$ \\
            \bottomrule
        \end{tabular}
    \end{subtable}
    \begin{subtable}{\linewidth}
        \caption{$\tilde\beta_1=10^6$, $\tilde\beta_2=10^5$}\label{tab:dose:65}
        \centering
        \begin{tabular}{ccccccc}
            \toprule 
            $\gamma/\gamma_0$                & $1.56\cdot 10^{-2}$ & $1.22\cdot 10^{-4}$ & $1.91\cdot 10^{-6}$ & $1.49\cdot 10^{-8}$ & $1.86\cdot 10^{-9}$ & $1.16 \cdot 10^{-10}$\\
            \midrule
            \#SSN                   & $1$              & $1$              & $3$              & $5$              & $7$              & $27$\\
            \% $\omega_R$ above $L$ & $0$           & $0$           & $15.56$           & $24.44$           & $24.44$           & $22.22$ \\
            \% $\omega_T$ below $U$ & $100$           & $100$           & $100$           & $29.03$           & $19.35$           & $19.35$ \\
            \bottomrule
        \end{tabular}
    \end{subtable}
    \begin{subtable}{\linewidth}
        \caption{$\tilde\beta_1=10^5$, $\tilde\beta_2=10^6$}\label{tab:dose:56}
        \centering
        \begin{tabular}{ccccccc}
            \toprule
            $\gamma/\gamma_0$                & $1.56\cdot 10^{-2}$ & $1.22\cdot 10^{-4}$ & $1.91\cdot 10^{-6}$ & $1.49\cdot 10^{-8}$ & $1.86\cdot 10^{-9}$ & $1.16 \cdot 10^{-10}$\\
            \midrule
            \#SSN                   & $1$              & $1$              & $2$              & $3$              & $5$              & $8$ \\
            \% $\omega_R$ above $L$ & $0$           & $0$           & $0$           & $4.44$              & $2.22$              & $2.22$ \\
            \% $\omega_T$ below $U$ & $100$           & $100$           & $100$           & $70.97$              & $58.06$              & $58.06$ \\
            \bottomrule
        \end{tabular}
    \end{subtable}
    \begin{subtable}{\linewidth}
        \caption{$\tilde\beta_1=10^6$, $\tilde\beta_2=10^6$}\label{tab:dose:66}
        \centering
        \begin{tabular}{ccccccc}
            \toprule
            $\gamma/\gamma_0$                & $1.56\cdot 10^{-2}$ & $1.22\cdot 10^{-4}$ & $1.91\cdot 10^{-6}$ & $1.49\cdot 10^{-8}$ & $1.86\cdot 10^{-9}$ & $1.16 \cdot 10^{-10}$\\
            \midrule
            \#SSN                   & $1$              & $1$              & $3$              & $10$              & $10$              & $*$ \\
            \% $\omega_R$ above $L$ & $0$           & $0$           & $13.33$           & $13.33$           & $11.11$           & $*$ \\
            \% $\omega_T$ below $U$ & $100$           & $100$           & $100$           & $19.35$           & $16.13$           & $*$ \\
            \bottomrule
        \end{tabular}
    \end{subtable}
\end{table}
\begin{table}[p]
    \caption{Convergence of semismooth Newton method for $\gamma=10^{-7}\gamma_0$: step length $\tau_k$ and achieved residual norm $\norm{T(u^k)}$ in each iteration $k$}
    \label{tab:dose:ssn}
    \centering
    \begin{tabular}{ccccccc}
        \toprule
        $k$ & $1$ & $2$ & $3$ & $4$ & $5$ & $6$   \\
        \midrule
        $\tau_k$ &$1.00$ & $1.00$ & $0.50$ & $0.125$ & $0.50$ & $0.50$  \\
        $\norm{T(u^k)}$ & $9.862\cdot10^{2}$ &   $9.862\cdot10^{2}$ &   $2.449\cdot10^{2}$ & $2.065\cdot10^{2}$ & $1.782\cdot10^{2}$ & $1.007\cdot10^{2}$ \\
        \midrule
        $k$ & $7$ & $8$ & $9$ & $10$ & $11$ & $12$   \\ 
        \midrule
        $\tau_k$ & $1.00$ & $1.00$ & $1.00$ & $1.00$ & $1.00$ & $1.00$  \\
        $\norm{T(u^k)}$ & $2.083\cdot10^{1}$ & $4.083\cdot10^{0}$ & $1.295\cdot10^{-1}$ & $8.216\cdot10^{-6}$ & $4.665\cdot10^{-11}$ & $1.724\cdot10^{-14}$ \\ 
        \bottomrule
    \end{tabular}
\end{table}

To illustrate the convergence behavior of the semismooth Newton method, \cref{tab:dose:ssn} shows the iteration history for $\tilde\beta_1=\tilde\beta_2=10^6$ and $\gamma=10^{-7}\gamma_0$ (without warmstarts). For each iteration $k$, the step length $\tau_k$ returned by the line search and the norm of the residual in the (regularized) optimality condition \eqref{eq:regoptsys_ref} are reported. For some initial steps, moderate damping of the semismooth Newton steps is required, leading to linear convergence. Starting from iteration $7$, full Newton steps are taken, and superlinear convergence can be observed. 

Finally, removing the central section of the risk region---so that now $\omega_R = [-0.7,-0.55] \cup [0.55,0.7]$---results in \eqref{eq:boxproblem} admitting a feasible solution. \Cref{fig:dose77} shows the dose profile and dose volume histogram for solving the regularized dose-constrained problem using $\tilde\beta_1=\tilde\beta_2=10^7$ and final $\gamma\approx 2\cdot 10^{-11}\gamma_0$.  We see that the solution satisfies the constraints in \eqref{eq:boxproblem} for sufficiently large $\beta$, as expected from \cref{thm:opt_limit_feas}.

\begin{figure}[t]
    \centering
    \begin{subfigure}[t]{0.495\textwidth}
        \input{Dose_77.tex}
        \caption{Final dose $C_\Omega y$,  risk level $L\chi_{\omega_R}$, target level $U\chi_{\omega_T}$}
    \end{subfigure}
    \hfill
    \begin{subfigure}[t]{0.495\textwidth}
        \begin{tikzpicture}

\begin{axis}[
width=\linewidth,
xlabel={Dose level},
ylabel={Volume fraction},
xmin=0, xmax=0.6,
ymin=0, ymax=1,
axis on top,
xtick distance=0.1,
legend style={legend cell align=left,align=left,draw=none,at={(0.5,1.03)},anchor=south,nodes={inner xsep=2ex}},
legend columns=4,
legend entries={$\omega_R$,$\omega_T$},
]
\addplot [name path=omR, mycolor2,thick]
table {%
0 1
0.00301507537688442 1
0.00603015075376884 1
0.00904522613065327 1
0.0120603015075377 1
0.0150753768844221 1
0.0180904522613065 0.947368421052632
0.021105527638191 0.894736842105263
0.0241206030150754 0.842105263157895
0.0271356783919598 0.789473684210526
0.0301507537688442 0.736842105263158
0.0331658291457286 0.736842105263158
0.0361809045226131 0.684210526315789
0.0391959798994975 0.684210526315789
0.0422110552763819 0.631578947368421
0.0452261306532663 0.631578947368421
0.0482412060301508 0.578947368421053
0.0512562814070352 0.578947368421053
0.0542713567839196 0.526315789473684
0.057286432160804 0.526315789473684
0.0603015075376884 0.526315789473684
0.0633165829145729 0.473684210526316
0.0663316582914573 0.473684210526316
0.0693467336683417 0.473684210526316
0.0723618090452261 0.421052631578947
0.0753768844221105 0.421052631578947
0.078391959798995 0.421052631578947
0.0814070351758794 0.368421052631579
0.0844221105527638 0.368421052631579
0.0874371859296482 0.368421052631579
0.0904522613065327 0.368421052631579
0.0934673366834171 0.315789473684211
0.0964824120603015 0.315789473684211
0.0994974874371859 0.315789473684211
0.10251256281407 0.315789473684211
0.105527638190955 0.263157894736842
0.108542713567839 0.263157894736842
0.111557788944724 0.263157894736842
0.114572864321608 0.263157894736842
0.117587939698492 0.263157894736842
0.120603015075377 0.210526315789474
0.123618090452261 0.210526315789474
0.126633165829146 0.210526315789474
0.12964824120603 0.210526315789474
0.132663316582915 0.210526315789474
0.135678391959799 0.210526315789474
0.138693467336683 0.157894736842105
0.141708542713568 0.157894736842105
0.144723618090452 0.157894736842105
0.147738693467337 0.157894736842105
0.150753768844221 0.157894736842105
0.153768844221106 0.157894736842105
0.15678391959799 0.105263157894737
0.159798994974874 0.105263157894737
0.162814070351759 0.105263157894737
0.165829145728643 0.105263157894737
0.168844221105528 0.105263157894737
0.171859296482412 0.105263157894737
0.174874371859296 0.105263157894737
0.177889447236181 0.0526315789473684
0.180904522613065 0.0526315789473684
0.18391959798995 0.0526315789473684
0.186934673366834 0.0526315789473684
0.189949748743719 0.0526315789473684
0.192964824120603 0.0526315789473684
0.195979899497487 0.0526315789473684
0.198994974874372 0.0526315789473684
0.202010050251256 0
0.205025125628141 0
0.208040201005025 0
0.21105527638191 0
0.214070351758794 0
0.217085427135678 0
0.220100502512563 0
0.223115577889447 0
0.226130653266332 0
0.229145728643216 0
0.232160804020101 0
0.235175879396985 0
0.238190954773869 0
0.241206030150754 0
0.244221105527638 0
0.247236180904523 0
0.250251256281407 0
0.253266331658291 0
0.256281407035176 0
0.25929648241206 0
0.262311557788945 0
0.265326633165829 0
0.268341708542714 0
0.271356783919598 0
0.274371859296482 0
0.277386934673367 0
0.280402010050251 0
0.283417085427136 0
0.28643216080402 0
0.289447236180905 0
0.292462311557789 0
0.295477386934673 0
0.298492462311558 0
0.301507537688442 0
0.304522613065327 0
0.307537688442211 0
0.310552763819095 0
0.31356783919598 0
0.316582914572864 0
0.319597989949749 0
0.322613065326633 0
0.325628140703518 0
0.328643216080402 0
0.331658291457286 0
0.334673366834171 0
0.337688442211055 0
0.34070351758794 0
0.343718592964824 0
0.346733668341709 0
0.349748743718593 0
0.352763819095477 0
0.355778894472362 0
0.358793969849246 0
0.361809045226131 0
0.364824120603015 0
0.3678391959799 0
0.370854271356784 0
0.373869346733668 0
0.376884422110553 0
0.379899497487437 0
0.382914572864322 0
0.385929648241206 0
0.38894472361809 0
0.391959798994975 0
0.394974874371859 0
0.397989949748744 0
0.401005025125628 0
0.404020100502513 0
0.407035175879397 0
0.410050251256281 0
0.413065326633166 0
0.41608040201005 0
0.419095477386935 0
0.422110552763819 0
0.425125628140703 0
0.428140703517588 0
0.431155778894472 0
0.434170854271357 0
0.437185929648241 0
0.440201005025126 0
0.44321608040201 0
0.446231155778894 0
0.449246231155779 0
0.452261306532663 0
0.455276381909548 0
0.458291457286432 0
0.461306532663317 0
0.464321608040201 0
0.467336683417085 0
0.47035175879397 0
0.473366834170854 0
0.476381909547739 0
0.479396984924623 0
0.482412060301508 0
0.485427135678392 0
0.488442211055276 0
0.491457286432161 0
0.494472361809045 0
0.49748743718593 0
0.500502512562814 0
0.503517587939698 0
0.506532663316583 0
0.509547738693467 0
0.512562814070352 0
0.515577889447236 0
0.518592964824121 0
0.521608040201005 0
0.524623115577889 0
0.527638190954774 0
0.530653266331658 0
0.533668341708543 0
0.536683417085427 0
0.539698492462311 0
0.542713567839196 0
0.54572864321608 0
0.548743718592965 0
0.551758793969849 0
0.554773869346734 0
0.557788944723618 0
0.560804020100502 0
0.563819095477387 0
0.566834170854271 0
0.569849246231156 0
0.57286432160804 0
0.575879396984925 0
0.578894472361809 0
0.581909547738693 0
0.584924623115578 0
0.587939698492462 0
0.590954773869347 0
0.593969849246231 0
0.596984924623116 0
0.6 0
};
\addplot[forget plot,name path=L,color=mycolor2,style=dashed] coordinates {(0.2,0) (0.2,1)};
\fill[intersection segments={of=L and omR},pattern=north east lines, pattern color=mycolor2] -- cycle;

\addplot [name path=omT,mycolor5,thick]
table {%
0 1
0.00301507537688442 1
0.00603015075376884 1
0.00904522613065327 1
0.0120603015075377 1
0.0150753768844221 1
0.0180904522613065 1
0.021105527638191 1
0.0241206030150754 1
0.0271356783919598 1
0.0301507537688442 1
0.0331658291457286 1
0.0361809045226131 1
0.0391959798994975 1
0.0422110552763819 1
0.0452261306532663 1
0.0482412060301508 1
0.0512562814070352 1
0.0542713567839196 1
0.057286432160804 1
0.0603015075376884 1
0.0633165829145729 1
0.0663316582914573 1
0.0693467336683417 1
0.0723618090452261 1
0.0753768844221105 1
0.078391959798995 1
0.0814070351758794 1
0.0844221105527638 1
0.0874371859296482 1
0.0904522613065327 1
0.0934673366834171 1
0.0964824120603015 1
0.0994974874371859 1
0.10251256281407 1
0.105527638190955 1
0.108542713567839 1
0.111557788944724 1
0.114572864321608 1
0.117587939698492 1
0.120603015075377 1
0.123618090452261 1
0.126633165829146 1
0.12964824120603 1
0.132663316582915 1
0.135678391959799 1
0.138693467336683 1
0.141708542713568 1
0.144723618090452 1
0.147738693467337 1
0.150753768844221 1
0.153768844221106 1
0.15678391959799 1
0.159798994974874 1
0.162814070351759 1
0.165829145728643 1
0.168844221105528 1
0.171859296482412 1
0.174874371859296 1
0.177889447236181 1
0.180904522613065 1
0.18391959798995 1
0.186934673366834 1
0.189949748743719 1
0.192964824120603 1
0.195979899497487 1
0.198994974874372 1
0.202010050251256 1
0.205025125628141 1
0.208040201005025 1
0.21105527638191 1
0.214070351758794 1
0.217085427135678 1
0.220100502512563 1
0.223115577889447 1
0.226130653266332 1
0.229145728643216 1
0.232160804020101 1
0.235175879396985 1
0.238190954773869 1
0.241206030150754 1
0.244221105527638 1
0.247236180904523 1
0.250251256281407 1
0.253266331658291 1
0.256281407035176 1
0.25929648241206 1
0.262311557788945 1
0.265326633165829 1
0.268341708542714 1
0.271356783919598 1
0.274371859296482 1
0.277386934673367 1
0.280402010050251 1
0.283417085427136 1
0.28643216080402 1
0.289447236180905 1
0.292462311557789 1
0.295477386934673 1
0.298492462311558 1
0.301507537688442 1
0.304522613065327 1
0.307537688442211 1
0.310552763819095 1
0.31356783919598 1
0.316582914572864 1
0.319597989949749 1
0.322613065326633 1
0.325628140703518 1
0.328643216080402 1
0.331658291457286 1
0.334673366834171 1
0.337688442211055 1
0.34070351758794 1
0.343718592964824 1
0.346733668341709 1
0.349748743718593 1
0.352763819095477 1
0.355778894472362 1
0.358793969849246 1
0.361809045226131 1
0.364824120603015 1
0.3678391959799 1
0.370854271356784 1
0.373869346733668 1
0.376884422110553 1
0.379899497487437 1
0.382914572864322 1
0.385929648241206 1
0.38894472361809 1
0.391959798994975 1
0.394974874371859 1
0.397989949748744 1
0.401005025125628 1
0.404020100502513 1
0.407035175879397 1
0.410050251256281 1
0.413065326633166 1
0.41608040201005 1
0.419095477386935 1
0.422110552763819 1
0.425125628140703 1
0.428140703517588 1
0.431155778894472 1
0.434170854271357 1
0.437185929648241 1
0.440201005025126 1
0.44321608040201 1
0.446231155778894 1
0.449246231155779 1
0.452261306532663 1
0.455276381909548 1
0.458291457286432 1
0.461306532663317 1
0.464321608040201 1
0.467336683417085 1
0.47035175879397 1
0.473366834170854 1
0.476381909547739 1
0.479396984924623 1
0.482412060301508 1
0.485427135678392 1
0.488442211055276 1
0.491457286432161 1
0.494472361809045 1
0.49748743718593 1
0.500502512562814 1
0.503517587939698 1
0.506532663316583 1
0.509547738693467 1
0.512562814070352 1
0.515577889447236 1
0.518592964824121 1
0.521608040201005 1
0.524623115577889 1
0.527638190954774 1
0.530653266331658 1
0.533668341708543 1
0.536683417085427 1
0.539698492462311 1
0.542713567839196 1
0.54572864321608 1
0.548743718592965 1
0.551758793969849 1
0.554773869346734 1
0.557788944723618 1
0.560804020100502 1
0.563819095477387 1
0.566834170854271 1
0.569849246231156 1
0.57286432160804 1
0.575879396984925 1
0.578894472361809 1
0.581909547738693 1
0.584924623115578 1
0.587939698492462 1
0.590954773869347 1
0.593969849246231 1
0.596984924623116 1
0.6 1
};
\addplot[name path=U,color=mycolor5,style=dashed] coordinates {(0.5,0) (0.5,1)};
\fill[intersection segments={of=omT and U,sequence={L1 -- R*}},pattern=north east lines, pattern color=mycolor5] -- cycle;
\end{axis}

\end{tikzpicture}
        \caption{Dose volume histogram for risk region $\omega_R$ and target region $\omega_T$}
    \end{subfigure}
    \caption{Dose information for feasible dose-penalized problem \eqref{eq:problem} with $\tilde\beta_1=10^7$, $\tilde\beta_2=10^7$}
    \label{fig:dose77}
\end{figure}

\section{Conclusions}\label{sec:concl}

Volumetric dose constraints arising in, e.g., radiotherapy treatment planning can be formulated using $L^1$ penalization. This leads to a non-differentiable optimal control problem for partial differential equations that can be analyzed and shown to be well-posed using tools from convex analysis. After introducing a Moreau--Yosida regularization, these problems can be solved efficiently by a semismooth Newton method together with a homotopy in the regularization parameter. Our numerical examples illustrate that this approach significantly outperforms formulations via pointwise state constraints, in particular with respect to the dose volume histograms commonly used to evaluate structure survival probabilities.

Natural next steps are the extension of the proposed approach to radiative transport equations---which are challenging both analytically and numerically due to their hyperbolic nature and their increased dimensionality (angular dependence)---and the application to concrete problems in radiotherapy treatment planning. 
Here we note that the analysis in \cref{sec:optcon} and \cref{sec:MYreg} only relies on the assumption that $S$ is a completely continuous affine operator between Hilbert spaces. Recent work on using realistic models for radiotherapy treatment \cite{FraHerSan10,TerKokFra16} has established the complete continuity of the relevant control-to-state operator.  
While in this case we cannot rely on the range assumption providing the norm gap needed to apply a semismooth Newton method, we point out that this can be replaced by including an additional smoothing step in the algorithm as in \cite{Ulbrich:2011}; the norm gap (and hence the range assumption) is also not required when directly considering the finite-dimensional discretized optimality conditions.
This is left for future work.

\printbibliography

\end{document}